\documentclass [11pt]{amsart}
\usepackage{amsmath,amssymb}
 \usepackage{amsthm, amsfonts}
 \usepackage{enumerate}
 \usepackage{amssymb,amsmath,amsthm, amsfonts}
 \usepackage{cite}
 \usepackage{xcolor}
 \usepackage[pagewise]{lineno}%\linenumbers
\newtheorem{theorem}{Theorem}[section]
\newtheorem{proposition}[theorem]{Proposition}
\newtheorem{lemma}[theorem]{Lemma}
\newtheorem{corollary}[theorem]{Corollary}
\theoremstyle{definition}

\theoremstyle{remark}

\numberwithin{equation}{section}

\begin{document}

\title [{{The Zariski topology on the secondary-like spectrum o}}]{The Zariski topology on the secondary like spectrum of a module}

 \author[{{S. Salam and K. Al-Zoubi }}]{Saif Salam and \textit{Khaldoun Al-Zoubi}* }
\address
{\textit{Saif Salam, Department of Mathematics and
Statistics, Jordan University of Science and Technology, P.O.Box
3030, Irbid 22110, Jordan.}}
\bigskip
{\email{\textit{smsalam19@sci.just.edu.jo}}}

\address
{\textit{ Khaldoun Al-Zoubi , Department of Mathematics and
Statistics, Jordan University of Science and Technology, P.O.Box
3030, Irbid 22110, Jordan.}}
\bigskip
{\email{\textit{kfzoubi@just.edu.jo}}}


\date{}
\begin{abstract}
  Let $R$ be a commutative ring with unity and $M$ be a left $R$-module. We define the secondary-like spectrum of $M$ to be the set of all secondary submodules $K$ of $M$ such that $Ann_R(soc(K))=\sqrt{Ann_R(K)}$, and we denote it by $Spec^L(M)$. In this paper, we introduce a topology on $Spec^L(M)$ having the Zariski topology on the second spectrum $Spec^s(M)$ as a subspace topology, and study several topological structures of this topology.
 \end{abstract}

\keywords{secondary-like spectrum, Zariski topology, second spectrum. \\
$*$ Corresponding author}
 \maketitle

%---------------------------------------------------------------------

%                Sec1:   Introduction
%-----------------------------------------------------------------------
 \section{Introduction and Preliminaries }
Throughout this paper all rings are commutative with identity and all modules are left unital modules. Let $M$ be an $R$-module. By $I\lhd R$ (resp. $N\leq M$) we mean that $I$ is an ideal of $R$ (resp. $N$ is a submodule of $M$). Let $M$ be an $R$-module, $N\leq M$ and $I\lhd R$. The annihilator of $N$ in $R$ (resp. the annihilator of $I$ in $M$) will be expressed by $Ann_R(N)=\{r\in R\,\mid\,rN=\{0\}\}$ (resp. $Ann_M(I)=\{m\in M\,\mid\, Im=\{0\}\}$). For a ring $R$, the set of all prime ideals of $R$ is denoted by $Spec(R)$ and it is called the prime spectrum of $R$. For any ideal $I$ of $R$, define $V^R(I)$ as the set of all prime ideals containing $I$, i.e. $V^R(I)=\{p\in Spec(R)\,\mid\, p\supseteq I\}$. Then there exists a topology on $Spec(R)$ having $\{V^R(I)\,\mid\, I\lhd R\}$ as the collection of closed sets. This topology is known as the Zariski topology on $Spec(R)$ (see, for example, \cite{atiyah1969introduction,bourbakialgebre}).

Let $M$ be an $R$-module. A submodule $S$ of $M$ is called second if $S\neq\{0\}$ and for every $r\in R$, we have $rS=S$ or $rS=0$. It is clear that $Ann_R(S)\in Spec(R)$ for any second submodule $S$ of $M$. The second spectrum of $M$, given by $Spec^s(M)$, is the set of all second submodules of $M$. If $Spec^s(M)=\emptyset$, then we say that $M$ is a secondless $R$-module. The socle of a submodule $N$ of $M$, expressed by $soc(N)$, is the sum of all second submodules of $M$. If there is no second submodules contained in $N$, then $soc(N)$ is defined to be $\{0\}$. The second submodules and the socle of submodules have been studied by many authors (see, for example, \cite{ansari2011dual,cceken2013dual,ansari2012dual,yassemi2001dual,ansari2013dual,ansari}). A submodule $K$ of $M$ is said to be secondary if $K\neq\{0\}$ and for every $r\in R$, we have $rK=K$ or there exists a positive integer $n$ such that $r^n K=\{0\}$. For any ideal $I$ of $R$, we denote the radical of $I$ by $\sqrt{I}$. It is easily seen that, if $K$ is a secondary submodule of $M$, then $Ann_R(K)$ is a primary ideal of $R$ and hence $\sqrt{Ann_R(K)}\in Spec(R)$. For more details concerning the secondary submodules, one can look in \cite{ansari2012secondary,ghaleb2021spectrum,ansari20162}.

Let $M$ be an $R$-module and let $\zeta^{s*}(M)=\{V^{s*}(N)\,\mid\,N\leq M\}$, where $V^{s*}(N)=\{S\in Spec^s(M)\,\mid\, N\supseteq S\}$ for any $N\leq M$. We say that $M$ is a cotop module if $\zeta^{s*}(M)$ is closed under finite union and in this case, $\zeta^{s*}(M)$ satisfies the axioms for closed sets of a topology on $Spec^{s}(M)$. The generated topology is called the quasi Zariski topology on $Spec^s(M)$. Unlike $\zeta^{s*}(M)$, $\zeta^{s}(M)=\{V^s(N)\,\mid\, N\leq M\}$ where $V^s(N)=\{S\in Spec^s(M)\,\mid\, Ann_R(S)\supseteq Ann_R(N)\}$ for any $N\leq M$ always satisfies all of the topology axioms for the closed sets. The resulting topology is called the Zariski topology on $Spec^s(M)$. When $Spec^s(M)\neq \emptyset$, the map $\psi:Spec^s(M)\rightarrow Spec(R/Ann_R(M))$ defined by $S\mapsto Ann_R(S)/Ann_R(M)$ for every $S\in Spec^s(M)$ is called the natural map of $Spec^s(M)$. For more information about the topologies on $Spec^s(M)$ and the natural map of $Spec^s(M)$, see \cite{ansari2014zariski,farshadifar2013modules,abuhlail2011dual,abuhlail2015zariski}. In addition, in \cite{salam2022zariski,salam2022graded}, the Zariski topology on the graded second spectrum of graded modules have been studied.

In this work, we call the set of all secondary submodules $K$ of an $R$-module $M$ satisfying the condition $Ann_R(soc(K))=\sqrt{Ann_R(K)}$ the secondary-like spectrum of $M$ and denote it by $Spec^L(M)$. It is clear that $Spec^s(M)\subseteq Spec^L(M)$. But the converse inclusion is not true in general. For example, if $M$ is a non-zero vector space over a field $F$, then $Spec^s(M)=Spec^L(M)=\{N\leq M\,\mid\, N\neq\{0\}\}$. On the other hand, if we take $\mathbb{Z}_8$ as $\mathbb{Z}_8$-module, then it is easy to see that the submodule $K=\{0, 2, 4, 6\}\in Spec^L(\mathbb{Z}_8)$. However, $K=\{0, 2, 4, 6\}\notin Spec^s(\mathbb{Z}_8)$ as $2K=\{0, 4\}\neq N$ and $2K\neq\{0\}$. Also, if we take $\mathbb{Z}$ as $\mathbb{Z}$-module, then $Spec^s(\mathbb{Z})=Spec^L(\mathbb{Z})=\emptyset$. For an $R$-module $M$, it is obvious that $soc(K)\neq \{0\}$ for any $K\in Spec^L(M)$ as $Ann_R(soc(K))=\sqrt{Ann_R(K)}\in Spec(R)$.

We start this work by introducing the notion of the secondary cotop module which is a generalization of the cotop module. For this, we define the variety of any submodule $N$ of an $R$-module $M$ by $\nu^{s*}(N)=\{K\in Spec^L(M)\,\mid\, soc(K)\subseteq N\}$ and we set $\Theta^{s*}(M)=\{\nu^{s*}(N)\,\mid\, N\leq M\}$. Then we say that $M$ is a secondary cotop module if $\Theta^{s*}(M)$ is closed under finite union. When this is the case, we see that $\Theta^{s*}(M)$ induces a topology and this topology is called the quasi-Zariski topology on $Spec^L(M)$ (Theorem \ref{Theorem 2.1}). Next, we introduce another variety for any $N\leq M$ by $\nu^{s}(N)=\{K\in Spec^L(M)\,\mid\, Ann_R(N)\subseteq \sqrt{Ann_R(K)}\}$ and prove that there always exists a topology on $Spec^L(M)$ having $\Theta^{s}(M)=\{\nu^{s}(N)\,\mid\, N\leq M\}$ as the collection of closed sets. We call this topology the Zariski topology on $Spec^L(M)$ or simply $\mathcal{SL}$-topology (Theorem \ref{Theorem 2.3}). We provide some relationships between the varieties $\nu^{s*}(N), \nu^{s}(N), V^{s*}(N)$ and $V^{s}(N)$ for any submodule $N$ of an $R$-module $M$ and conclude that every secondary cotop module is a cotop module (Lemma \ref{Lemma 2.4} and Corollary \ref{Corollary 2.5}). In addition, using these relations, we see that the Zariski topology on $Spec^s(M)$ for an $R$-module $M$ is a topological subspace of $Spec^L(M)$ equipped with the $\mathcal{SL}$-topology. Next, for an $R$-module $M$, we introduce the map $\varphi:Spec^L(M)\rightarrow Spec(R/Ann_R(M))$ by $\varphi(K)=\sqrt{Ann_R(K)}/Ann_R(M)$ to obtain some relations between the properties of $Spec^L(M)$ and $Spec(R/Ann_R(M))$. For example, we relate the connectedness of $Spec^L(M)$ with the connecteness of both $Spec(R/Ann_R(M))$ and $Spec^s(M)$ using $\varphi$ and $\psi$ (Theorem \ref{theorem 2.12}). In section 2, among other results, we show that the secondary-like spectrums of two modules are homeomorphic if there exists a module isomorphism between these modules (Corollary \ref{corollary 2.15}). In section 3, we introduce a base for the Zariski topology on $Spec^L(M)$ (Theorem \ref{theorem 3.1}). Also, we prove that the surjectivity of $\varphi$ implies that the basic open sets are quasi-compact and conclude that $Spec^L(M)$ is quasi-compact (Theorem \ref{Theorem 3.4}). In section 4, we investigate the irreducibility of $Spec^L(M)$ with respect to the $\mathcal{SL}$-topology. In particular, it is shown that if $\varphi$ is surjective, then every irreducible closed subset of $Spec^L(M)$ has a generic point and we determine exactly the set of all irreducible components of $Spec^L(M)$ (Theorem \ref{theorem 4.5} and Corollary \ref{corollary 4.8}). Furthermore, we study $Spec^L(M)$ from the viewpoint of being $T_0$-space, $T_1$-space and spectral space.
\section{The Zariski topology on $Spec^L(M)$}
In this section, we define two varieties of submodules and we use their properties to construct the quasi-Zariski topology and the $\mathcal{SL}$-topology on $Spec^L(M)$. In addition, we introduce some relationships between $Spec^L(M)$, $Spec(R/Ann_R(M))$ and $Spec^s(M)$.
%----------------------------Theorem 2.1--------------------
%-----------------------------------------------------
%----------------------------------------------------
\begin{theorem}\label{Theorem 2.1}
Let $M$ be an $R$-module. For any submodule $N$ of $M$, we define the variety of $N$ by $\nu^{s*}(N)=\{K\in Spec^L(M)\,\mid\, soc(K)\subseteq N\}$. Then we have the following:
\begin{enumerate}
\item $\nu^{s*}(M)=Spec^L(M)$ and $\nu^{s*}(0)=\emptyset$.
\item $\underset{j\in \Delta}{\bigcap} \nu^{s*}(N_j)=\nu^{s*}(\underset{j\in \Delta}{\bigcap}N_j)$ for any family of submodules $\{N_j\}_{j\in\Delta}$ and any index set $\Delta$.
\item $\nu^{s*}(N)\cup \nu^{s*}(L)\subseteq \nu^{s*}(N+L)$ for any $N, L\leq M$.
\item If $N_1, N_2\leq M$ with $N_1\subseteq N_2$, then $\nu^{s*}(N_1)\subseteq \nu^{s*}(N_2)$.
\item $\nu^{s*}(soc(N))=\nu^{s*}(N)$ for any submodule $N$ of $M$.
\end{enumerate}
\end{theorem}
\begin{proof}
The proof is obvious.
\end{proof}
The reverse inclusion in Theorem \ref{Theorem 2.1}(3) is not true in general. For example, take $M=\mathbb{Z}_2\times\mathbb{Z}_2$ as $\mathbb{Z}$-module and let $N=\mathbb{Z}_2\times\{0\}$ and $L=\{0\}\times\mathbb{Z}_2$. Then $\nu^{s*}(N+L)=\nu^{s*}(M)=Spec^L(M)$. Note that $M\in Spec^s(M)\subseteq Spec^L(M)$. It follows that $M\in \nu^{s*}(N+L)$ and $soc(M)=M$. But $M\nsubseteq N$ and $M\nsubseteq L$. Therefore, $M\notin \nu^{s*}(N)\cup \nu^{s*}(L)$.

Let $M$ be an $R$-module. Using Theorem \ref{Theorem 2.1}(1), (2) and (3), we can see that there exists a topology on $Spec^L(M)$ having $\Theta^{s*}(M)=\{\nu^{s*}(N)\,\mid\, N\leq M\}$ as the collection of closed sets if and only if $\Theta^{s*}(M)$ is closed under finite union. When this is the case, we call the resulting topology the quasi-Zariski topology on $Spec^L(M)$ and we call the module $M$ a secondary cotop module.

Recall that an $R$-module $M$ is said to be comultiplication if any submodule $N$ of $M$ has the form $Ann_M(I)$ for some ideal $I$ of $R$. By \cite[Lemma 3.7]{ansari2007dual}, if $N$ is a submodule of a comultiplication $R$-module $M$, then $N=Ann_M(Ann_R(N))$.
%--------------------------------------------------------------
%-------------------------------theorem 2.2------------------
%------------------------------------------------------
\begin{theorem}
Every comultiplication module is a secondary cotop module.
\end{theorem}
\begin{proof}
Let $M$ be a comultiplication $R$-module and $N_1, N_2\leq M$. It is sufficient to show that $\nu^{s*}(N_1+N_2)\subseteq\nu^{s*}(N_1)+\nu^{s*}(N_2)$. So let $K\in \nu^{s*}(N_1+N_2)$. Then $soc(K)\subseteq N_1+N_2$. Hence $Ann_R(N_1)\cap Ann_R(N_2)=Ann_R(N_1+N_2)\subseteq Ann_R(soc(K))=\sqrt{Ann_R(K)}\in Spec(R)$ as $K$ is a secondary submodule. This follows that $Ann_R(N_1)\subseteq Ann_R(soc(K))$ or $Ann_R(N_2)\subseteq Ann_R(soc(K))$. Thus $soc(K)\subseteq Ann_M(Ann_R(soc(K)))\subseteq Ann_M(Ann_R(N_1))=N_1$ or $soc(K)\subseteq Ann_M(Ann_R(soc(K)))\subseteq Ann_M(Ann_R(N_2))=N_2$. Therefore, $K\in\nu^{s*}(N_1)\cup\nu^{s*}(N_2)$.
\end{proof}
In the following theorem, we construct the Zariski topology on $Spec^L(M)$ of an $R$-module $M$ by introducing another variety of any submodule $N$ of $M$ by  $\nu^{s}(N)=\{K\in Spec^L(M)\,\mid\, Ann_R(N)\subseteq \sqrt{Ann_R(K)}\}$.
%-----------------------------------------Theorem 2.3----------------------------------------
%-----------------------------------------------------------------------------
\begin{theorem}\label{Theorem 2.3}
The following hold for any $R$-module $M$:
\begin{enumerate}
\item $\nu^{s}(M)=Spec^L(M)$ and $\nu^{s}(0)=\emptyset$.
\item $\underset{i\in I}{\bigcap}\nu^{s}(N_i)=\nu^{s}(\underset{i\in I}{\bigcap}Ann_M(Ann_R(N_i)))$ for any family of submodules $\{N_i\}_{i\in I}$.
\item $\nu^{s}(N_1)\cup \nu^{s}(N_2)=\nu^{s}(N_1+N_2)$ for any $N_1, N_2\leq M$.
\end{enumerate}
\end{theorem}
\begin{proof}
(1) is straightforward.\\
(2) Let $K\in \underset{i\in I}{\bigcap}\nu^{s}(N_i)$. Then $Ann_R(N_i)\subseteq\sqrt{Ann_R(K)}$ and thus $Ann_M(\sqrt{Ann_R(K)})\subseteq Ann_M(Ann_R(N_i))$ for each $i\in I$. Hence $Ann_M(\sqrt{Ann_R(K)})\subseteq \underset{i\in I}{\bigcap}Ann_M(Ann_R(N_i))$ which follows that $Ann_R(\underset{i\in I}{\bigcap}Ann_M(Ann_R(N_i)))\subseteq Ann_R(Ann_M(\sqrt{Ann_R(K)}))=Ann_R(Ann_M(Ann_R(soc(K))))=Ann_R(soc(K))=\sqrt{Ann_R(K)}$. So $K\in \nu^{s}(\underset{i\in I}{\bigcap}Ann_M(Ann_R(N_i)))$ and so $\underset{i\in I}{\bigcap}\nu^{s}(N_i)\subseteq\nu^{s}(\underset{i \in I}{\bigcap}Ann_M(Ann_R(N_i)))$. Conversely, let $K\in\nu^{s}(\underset{i\in I}{\bigcap}Ann_M(Ann_R(N_i)))$. Then we have $Ann_R(\underset{i\in I}{\bigcap}Ann_M(Ann_R(N_i)))\subseteq\sqrt{Ann_R(K)}$. But for any $i\in I$, $\underset{i\in I}{\bigcap}Ann_M(Ann_R(N_i))\subseteq Ann_M(Ann_R(N_i))$ and thus $Ann_R(N_i)=Ann_R(Ann_M(Ann_R(N_i)))\subseteq Ann_R(\underset{i\in I}{\bigcap}Ann_M(Ann_R(N_i)))\subseteq\sqrt{Ann_R(K)}$. Therefore, $K\in\underset{i\in I}{\bigcap}\nu^{s}(N_i)$, as desired. \\
(3) Since $N_1\subseteq N_1+N_2$ and $N_2\subseteq N_1+N_2$, then $\nu^{s}(N_1)\subseteq\nu^{s}(N_1+N_2)$ and $\nu^{s}(N_2)\subseteq\nu^{s}(N_1+N_2)$. Therefore $\nu^s(N_1)\cup\nu^{s}(N_2)\subseteq\nu^s(N_1+N_2)$. Conversely, let $K\in\nu^{s}(N_1+N_2)$. Then $Ann_R(N_1)\cap Ann_R(N_2)=Ann_R(N_1+N_2)\subseteq\sqrt{Ann_R(K)}\in Spec(R)$ and hence $Ann_R(N_1)\subseteq\sqrt{Ann_R(K)}$ or $Ann_R(N_2)\subseteq\sqrt{Ann_R(K)}$. This implies that $K\in\nu^{s}(N_1)\cup\nu^{s}(N_2)$.
\end{proof}
Let $M$ be an $R$-module. By Theorem \ref{Theorem 2.3}, there always exists a topology on $Spec^L(M)$ having $\Theta^{s}(M)=\{\nu^{s}(N)\,\mid\,N\leq M\}$ as the collection of closed sets. We call this topology the Zariski topology on $Spec^L(M)$, or the $\mathcal{SL}$-topology for short.

In the following lemma, we state some useful relationships between the varieties $V^{s*}(N), V^{s}(N), \nu^{s*}(N)$ and $\nu^{s}(N)$.
%---------------------------------------------------------
%-------------------------------Lemma 2.4-------------------------
%------------------------------------------------------------------
\begin{lemma}\label{Lemma 2.4}
Let $N$ and $N^\prime$ be submodules of an $R$-module $M$ and $I$ be an ideal of $R$. Then the following hold:
\begin{enumerate}
\item $V^s(N)=\nu^s(N)\cap Spec^s(M)$.
\item $V^{s*}(N)=\nu^{s*}(N)\cap Spec^s(M)$.
\item If $\sqrt{Ann_R(N)}=\sqrt{Ann_R(N^\prime)}$, then $\nu^{s}(N)=\nu^s(N^\prime)$. The converse is also true if $N, N^\prime\in Spec^L(M)$.
\item $\nu^{s*}(N)\subseteq \nu^{s}(N)$. Furthermore, if $M$ is a comultiplication module, then the equality holds.
\item $\nu^s(Ann_M(I))=\nu^s(Ann_M(\sqrt{I}))=\nu^{s*}(Ann_M(I))=\nu^{s*}(Ann_M(\sqrt{I}))$.
\item $\nu^s(N)=\nu^s(Ann_M(Ann_R(N)))=\nu^s(Ann_M(\sqrt{Ann_R(N)}))=\nu^{s*}(Ann_M(Ann_R(N)))=\nu^{s*}(Ann_M(\sqrt{Ann_R(N)}))$.
\item If $N\in Spec^L(M)$ or $M$ is a comultiplication module, then $\nu^{s}(N)=\nu^s(soc(N))$.
\end{enumerate}
\end{lemma}
\begin{proof}
(1), (2), (3) and (4) are clear.\\
(5) We first show that $\nu^s(Ann_M(J))=\nu^{s*}(Ann_M(J))$ for any ideal $J$ of $R$. So let $J$ be an ideal of $R$. By (4), $\nu^{s*}(Ann_M(J))\subseteq \nu^{s}(Ann_M(J))$. For the reverse inclusion, let $K\in\nu^{s}(Ann_M(J))$. Then $Ann_R(Ann_M(J))\subseteq\sqrt{Ann_R(K)}=Ann_R(soc(K))$ and thus $soc(K)\subseteq Ann_M(Ann_R(soc(K)))\subseteq Ann_M(Ann_R(Ann_M(J)))=Ann_M(J)$ which follows that $K\in\nu^{s*}(Ann_M(J))$. Therefore, $\nu^{s}(Ann_M(J))\subseteq\nu^{s*}(Ann_M(J))$. Now, it is sufficient to prove that $\nu^{s*}(Ann_M(I))=\nu^{s*}(Ann_M(\sqrt{I}))$. Since $I\subseteq\sqrt{I}$, then we get $Ann_M(\sqrt{I})\subseteq Ann_M(I)$ which implies that $\nu^{s*}(Ann_M(\sqrt{I}))\subseteq\nu^{s*}(Ann_M(I))$ by Theorem \ref{Theorem 2.1}(4). Conversely, let $K\in\nu^{s*}(Ann_M(I))$. Then we have $soc(K)\subseteq Ann_M(I)$ and hence we obtain $I\subseteq Ann_R(Ann_M(I))\subseteq Ann_R(soc(K))=\sqrt{Ann_R(K)}$ which implies that $\sqrt{I}\subseteq Ann_R(soc(K))$. Thus $soc(K)\subseteq Ann_M(Ann_R(soc(K)))\subseteq Ann_M(\sqrt{I})$. Therefore $K\in\nu^{s*}(Ann_M(\sqrt{I}))$, as needed. \\
(6) Note that for any $K\in Spec^L(M)$, we have $K\in\nu^{s}(N)\Leftrightarrow Ann_R(Ann_M(Ann_R(N)))=Ann_R(N)\subseteq\sqrt{Ann_R(K)}\Leftrightarrow K\in\nu^{s}(Ann_M(Ann_R(N)))$ which implies that $\nu^{s}(N)=\nu^{s}(Ann_M(Ann_R(N)))$. Now, the result is clear by (5).\\
(7) First, suppose that $M$ is a comultiplication module. By Theorem \ref{Theorem 2.1}(5), we have $\nu^{s*}(soc(N))=\nu^{s*}(N)$. So we obtain $\nu^s(N)=\nu^{s*}(N)=\nu^{s*}(soc(N))=\nu^s(soc(N))$ by (4). For the second case, suppose that $N\in Spec^L(M)$. Since $soc(N)\subseteq N$, then $\nu^s(soc(N))\subseteq\nu^s(N)$. Let $K\in\nu^{s}(N)$ and it is enough to show that $Ann_R(soc(K))\subseteq\sqrt{Ann_R(K)}$. Since $K\in\nu^{s}(N)$, then $Ann_R(N)\subseteq\sqrt{Ann_R(K)}$. This follows that $Ann_R(soc(N))=\sqrt{Ann_R(N)}\subseteq \sqrt{\sqrt{Ann_R(K)}}=\sqrt{Ann_R(K)}$, as desired.
\end{proof}
Remark that the reverse inclusion in Lemma \ref{Lemma 2.4}(4) is not always true. For example, take $M=\mathbb{Q}\times\mathbb{Q}$ as $\mathbb{Q}$-module, where $\mathbb{Q}$ is the field of rational numbers. Let $N=\mathbb{Q}\times\{0\}$ and $L=\{0\}\times\mathbb{Q}$. By some computations, we can see that $L\in\nu^{s}(N)-\nu^{s*}(N)$.
%--------------------------------------------------
%------------------------Corollary 2.5-----------------
%------------------------------------------------------
\begin{corollary}\label{Corollary 2.5}
Every secondary cotop module is a cotop module.
\end{corollary}
\begin{proof}
Let $M$ be a secondary cotop module and $N_1, N_2\leq M$. By Lemma \ref{Lemma 2.4}(2), we have $V^{s*}(N_1)\cup V^{s*}(N_2)=(Spec^s(M)\cap\nu^{s*}(N_1))\cup (Spec^s(M)\cap\nu^{s*}(N_2))=Spec^s(M)\cap (\nu^{s*}(N_1)\cup\nu^{s*}(N_2))=Spec^s(M)\cap\nu^{s*}(T)=V^{s*}(T)$ for some submodule $T$ of $M$. Therefore $M$ is a cotop module.
\end{proof}
Let $M$ be an $R$-module. By Corollary \ref{Corollary 2.5}, if $M$ is a secondary cotop module, then $\zeta^{s*}(M)=\{V^{s*}(N)\,\mid\, N\leq M\}$ generates the quasi-Zariski topology on $Spec^s(M)$. By Lemma \ref{Lemma 2.4}(2), this topology is a topological subspace of the quasi-Zariski topology on $Spec^L(M)$. In addition, by Lemma \ref{Lemma 2.4}(1), $Spec^L(M)$ with the $\mathcal{SL}$-topology contains $Spec^s(M)$ with the Zariski topology as a topological subspace.

Let $M$ be an $R$-module. Throughout the rest of this paper, we assume that both $Spec^s(M)$ and $Spec^L(M)$ are non-empty sets unless stated otherwise, and are equipped with the Zariski topology. We also consider $\psi$ and $\varphi$ as defined in the introduction.

Let $M$ be an $R$-module and $p$ be a prime ideal of $R$. we set $Spec^L_p(M)=\{K\in Spec^L(M)\,\mid\,\sqrt{Ann_R(K)}=p\}$.
%-------------------------------------------------------------------------
%------------------------------------Proposition 2.6-----------------------
%-------------------------------------------------------------------------------
\begin{proposition} \label{Proposition 2.6}
The following statements are equivalent for any $R$-module $M$:
\begin{enumerate}
\item If whenever $K, K^\prime\in Spec^L(M)$ with $\nu^s(K)=\nu^s(K^\prime)$, then $K=K^\prime$.
\item $|Spec^L_p(M)|\leq 1$ for any $p\in Spec(R)$.
\item $\varphi$ is injective.
\end{enumerate}
\end{proposition}
\begin{proof}
(1)$\Rightarrow$(2): Let $p\in spec(R)$ and $K, K^\prime\in Spec^L_p(M)$. Then $K, K^\prime\in Spec^L(M)$ and $\sqrt{Ann_R(K)}=\sqrt{Ann_R(K^\prime)}=p$. By Lemma \ref{Lemma 2.4}(3), we obtain $\nu^s(K)=\nu^s(K^\prime)$ and thus $K=K^\prime$ by the assumption (1).\\
(2)$\Rightarrow$(3): Suppose that $\varphi(K)=\varphi(K^\prime)$, where $K, K^\prime\in Spec^L(M)$. Then $\sqrt{Ann_R(K)}=\sqrt{Ann_R(K^\prime)}$. Let $p=\sqrt{Ann_R(K)}\in Spec(R)$. Then $K, K^\prime\in Spec^L_p(M)$ and by the hypothesis, we get $K=K^\prime$.\\
(3)$\Rightarrow$(1): Let $K, K^\prime\in Spec^L(M)$ with $\nu^s(K)=\nu^s(K^\prime)$. By Lemma \ref{Lemma 2.4}(3), we have $\sqrt{Ann_R(K)}=\sqrt{Ann_R(K^\prime)}$. So $\varphi(K)=\varphi(K^\prime)$ and so $K=K^\prime$ as $\varphi$ is injective.
\end{proof}
The following is an easy result for Proposition \ref{Proposition 2.6}.
\begin{corollary}
Let $M$ be an $R$-module. If $|Spec^L_p(M)|= 1$ for any $p\in Spec(R)$, then $\varphi$ is bijective.
\end{corollary}
Let $M$ be an $R$-module. From now on, for any ideal $I$ of $R$ containing $Ann_R(M)$, $\overline{I}$ and $\overline{R}$ will denote $I/Ann_R(M)$ and $R/Ann_R(M)$, respectively. In addition, the class $r+Ann_R(M)$ in $R/Ann_R(M)$ will be expressed by $\overline{r}$ for any $r\in R$.

In the following lemma, we list some properties of the natural map $\psi$ of $Spec^L(M)$ which are needed in the rest of this paper.
%-----------------------------------------------------------------------------
%------------------------------Lemma 2.8----------------------------------
%----------------------------------------------------------------------------
\begin{lemma} (\cite[Proposition 3.6 and Theorem 3.11]{ansari2014zariski}) \label{Lemma 2.8} The following hold for any $R$-module $M$:
\begin{enumerate}
\item The natural map $\psi$ of $Spec^s(M)$ is continuous and $\psi^{-1}(V^{\overline{R}}(\overline{I}))=V^s(Ann_M(I))$ for any ideal $I$ of $R$ containing $Ann_R(M)$.
\item If $\psi$ is surjective, then $\psi$ is both open and closed with $\psi(V^s(N))=V^{\overline{R}}(\overline{Ann_R(N)})$ and $\psi(Spec^s(M)-V^s(N))=Spec(\overline{R})-V^{\overline{R}}(Ann_R(N))$ for any submodule $N$ of $M$.
\end{enumerate}
\end{lemma}
In the next two propositions, we provide similar results for $\varphi$.
%------------------------------------------------------------------------------------------
%------------------------Proposition 2.9----------------------------------------------
%-----------------------------------------------------------------------------------------
\begin{proposition}\label{Proposition 2.9}
Let $M$ be an $R$-module. Then $\varphi^{-1}(V^{\overline{R}}(\overline{I}))=\nu^s(Ann_M(I))$ for any ideal $I$ of $R$ containing $Ann_R(M)$. Therefore, $\varphi$ is continuous.
\end{proposition}
\begin{proof}
For any $K\in Spec^L(M)$, we have $K\in\varphi^{-1}(V^{\overline{R}}(\overline{I}))\Leftrightarrow \varphi(K)=\overline{\sqrt{Ann_R(K)}}\in V^{\overline{R}}(\overline{I})\Leftrightarrow\overline{I}\subseteq \overline{\sqrt{Ann_R(K)}}\Leftrightarrow I\subseteq \sqrt{Ann_R(K)}=Ann_R(soc(K))\Leftrightarrow soc(K)\subseteq Ann_M(I)\Leftrightarrow Ann_R(Ann_M(I))\subseteq Ann_R(soc(K))\Leftrightarrow K\in\nu^s(Ann_M(I))$. Therefore, $\varphi^{-1}(V^{\overline{R}})=\nu^s(Ann_M(I))$.
\end{proof}
%----------------------------------------------------------------------------------
%-----------------------------------Proposition 2.10---------------------
%---------------------------------------------------------------------------------
\begin{proposition}\label{Proposition 2.10}
Let $M$ be an $R$-module. If $\varphi$ is surjective, then $\varphi$ is both closed and open. In particular, $\varphi(\nu^s(N))=V^{\overline{R}}(\overline{Ann_R(N)})$ and $\varphi(Spec^L(M)-\nu^s(N))=Spec(\overline{R})-V^{\overline{R}}(\overline{Ann_R(N)})$ for any submodule $N$ of $M$.
\end{proposition}
\begin{proof}
Using Proposition \ref{Proposition 2.9}, we obtain $\varphi^{-1}(V^{\overline{R}}(\overline{I}))=\nu^s(Ann_M(I))$ for any ideal $I$ of $R$ containing $Ann_R(M)$ which follows that $\varphi^{-1}(V^{\overline{R}}(\overline{Ann_R(N)}))=\nu^s(Ann_M(Ann_R(N)))=\nu^s(N)$ by Lemma \ref{Lemma 2.4}(6). Since $\varphi$ is surjective, then $\varphi(\nu^s(N))=V^{\overline{R}}(\overline{Ann_R(N)})$. Hence $\varphi$ is closed. To show that $\varphi$ is an open map, note that $Spec^L(M)-\nu^s(N)=Spec^L(M)-\varphi^{-1}(V^{\overline{R}}(\overline{Ann_R(N)}))=\varphi^{-1}(Spec(\overline{R})-V^{\overline{R}}(\overline{Ann_R(N)}))$. Again, by the surjectivity of $\varphi$, we obtain $\varphi(Spec^L(M)-\nu^s(N))=Spec(\overline{R})-V^{\overline{R}}(\overline{Ann_R(N)})$, as desired.
\end{proof}
The proof of the following result is straightforward using Proposition \ref{Proposition 2.9} and Proposition \ref{Proposition 2.10}.
%-----------------------------------------------------------------------------------------
%-----------------------------------Corollary 2.11---------------------------------
%---------------------------------------------------------------------------------------
\begin{corollary}\label{corollary 2.11}
Let $M$ be an $R$-module. Then $\varphi$ is bijective if and only if $\varphi$ is a homeomorphism.
\end{corollary}
In the following Theorem, we relate the connectedness of $Spec^s(M)$ with the connectedness of both $Spec(\overline{R})$ and $Spec^s(M)$ for any $R$-module $M$ using the properties of $\varphi$ and $\psi$ that are stated in Lemma \ref{Lemma 2.8}, Proposition \ref{Proposition 2.9} and Proposition \ref{Proposition 2.10}.
%-------------------------------------------------------------------
%---------------------------------------------------Theorem 2.12---------------------------
%---------------------------------------------------------------------------------
\begin{theorem}\label{theorem 2.12}
Let $M$ be an $R$-module and consider the following statements:
\begin{enumerate}
\item $Spec^s(M)$ is connected.
\item $Spec^L(M)$ is connected.
\item $Spec(\overline{R})$ is connected.
\item $\overline{0}$ and $\overline{1}$ are the only idempotent elements of $\overline{R}$.
\end{enumerate}
(i) If $\varphi$ is surjective, then (1)$\Rightarrow$(2)$\Leftrightarrow$(3)$\Leftrightarrow$(4).\\
(ii) If $\psi$ is surjective, then all the four statements are equivalent.
\end{theorem}
\begin{proof} (i) (1)$\Rightarrow$(2): Suppose that $Spec^s(M)$ is connected. If $Spec^L(M)$ is disconnected, then there exists $W$ clopen in $Spec^L(M)$ such that $W\neq Spec^L(M)$ and $W\neq\emptyset$. Since $W$ is clopen in $Spec^L(M)$, then $W=Spec^L(M)-\nu^s(N)=\nu^s(N^\prime)$ for some submodules $N$ and $N^\prime$ of $M$. By Proposition \ref{Proposition 2.10},  $\varphi(W)$ is clopen in $Spec(\overline{R})$. But by Lemma \ref{Lemma 2.8}(1), $\psi$ is continuous. So $\psi^{-1}(\varphi(W))$ is clopen in $Spec^s(M)$ and thus $\psi^{-1}(\varphi(W))=\emptyset$ or $\psi^{-1}(\varphi(W))=Spec^s(M)$. If $\psi^{-1}(\varphi(W))=\emptyset$, then $\psi^{-1}(\varphi(\nu^s(N^\prime)))=\emptyset$ and hence $\psi^{-1}(V^{\overline{R}}(\overline{Ann_R(N^\prime)}))=\emptyset$ which follows that $V^s(N^\prime)=V^s(Ann_M(Ann_R(N^\prime)))=\emptyset$ using \cite[Lemma 3.3(b)]{farshadifar2013modules} and Lemma \ref{Lemma 2.8}(1). This implies that $Ann_R(N^\prime)\nsubseteq Ann_R(S)$ for any $S\in Spec^s(M)$. Since $W=\nu^s(N^\prime)\neq\emptyset$, then there exists $K\in Spec^L(M)$ such that $Ann_R(N^\prime)\subseteq\sqrt{Ann_R(K)}=Ann_R(soc(K))$. As $soc(K)\neq\{0\}$, then $\exists S^\prime\in Spec^s(M)$ such that $S^\prime\subseteq K$. Hence $\sqrt{Ann_R(K)}\subseteq\sqrt{Ann_R(S^\prime)}=Ann_R(S^\prime)\in Spec(R)$ and thus $Ann_R(N^\prime)\subseteq Ann_R(S^\prime)$, a contradiction. Now, if $\psi^{-1}(\varphi(W))=Spec^s(M)$, then $Spec^s(M)=\psi^{-1}(\varphi(\nu^s(N^\prime)))=\psi^{-1}(V^{\overline{R}}(\overline{Ann_R(N^\prime)}))=V^s(Ann_M(Ann_R(N^\prime)))=V^s(N^\prime)$ by Lemma \ref{Lemma 2.8}(1) and \cite[Lemma 3.3(b)]{farshadifar2013modules}, again. But by Lemma \ref{Lemma 2.4}(1), we have $V^s(N^\prime)=Spec^s(M)\cap\nu^s(N^\prime)$ which follows that $Spec^s(M)\subseteq\nu^s(N^\prime)=W=Spec^L(M)-\nu^s(N)$. Since $W=Spec^L(M)-\nu^s(N)\neq Spec^L(M)$, then $\exists K\in Spec^L(M)$ such that $Ann_R(N)\subseteq\sqrt{Ann_R(K)}$. As $Soc(K)\neq\{0\}$, then $\exists S\in Spec^s(M)$ such that $S\subseteq K$. So we have $Ann_R(N)\subseteq\sqrt{Ann_R(K)}\subseteq\sqrt{Ann_R(S)}$ and so $S\in\nu^s(N)$. But $S\in Spec^s(M)\subseteq Spec^L(M)-\nu^s(N)$. Therefore $S\notin\nu^s(N)$, a contradiction. Hence $Spec^L(M)$ is connected.\\
(2)$\Rightarrow$(3): Follows from the fact that the continuous image of a connected space is also connected.\\
(3)$\Rightarrow$(2): Suppose the contrary, i.e. $Spec^L(M)$ is disconnected. Then there exists a clopen set $W$ in $Spec^L(M)$ such that $W\neq\emptyset$ and $W\neq Spec^L(M)$.  By Proposition \ref{Proposition 2.10}, $\varphi(W)$ is clopen in $Spec(\overline{R})$ which is a connected space by the assumption (3). This implies that $\varphi(W)=Spec(\overline{R})$ or $\varphi(W)=\emptyset$. Since $W$ is open in $Spec^L(M)$, then $W=Spec^L(M)-\nu^s(N)$ for some $N\leq M$. Note that if $\varphi(W)=Spec(\overline{R})$, then $Spec(\overline{R})=\varphi(Spec^L(M)-\nu^s(N))=Spec(\overline{R})-V^{\overline{R}}(\overline{Ann_R(N)})$. Thus $V^{\overline{R}}(\overline{Ann_R(N)})=\emptyset$ and hence $\emptyset=\varphi^{-1}(\emptyset)=\varphi^{-1}(V^{\overline{R}}(\overline{Ann_R(N)}))=\nu^s(N)$ by Proposition \ref{Proposition 2.9}. This implies that $\nu^s(N)=\emptyset$ and hence $W=Spec^L(M)$, a contradiction. Now, if $\varphi(W)=\emptyset$, then $W\subseteq\varphi^{-1}(\varphi(W))=\emptyset$ and thus $W=\emptyset$ which is a contradiction. Consequently, $Spec^L(M)$ is connected.\\
The equivalence of (3) and (4) follows from \cite[p.132, Corollary 2 to Proposition 15]{bourbakialgebre}.\\
(ii) It is easy to see that if $\psi$ is surjective, then $\varphi$ is surjective and thus (1)$\Rightarrow$(2)$\Leftrightarrow$(3)$\Leftrightarrow$(4) by part (i). To complete the proof, it is enough to show that (4)$\Rightarrow$(1). But this follows from \cite[Corollary 3.13]{ansari2014zariski}.
\end{proof}
Let $f:M\rightarrow M^\prime$ be a module monomorphism of $R$-modules. Let $N\leq M$ and $N^\prime\leq M^\prime$ such that $N^\prime\subseteq f(M)$. By Lemma \cite[Lemma 2.19]{ansari20162}, we have $f(soc(N))=soc(f(N))$ and $f^{-1}(soc(N^\prime))=soc(f^{-1}(N^\prime))$. In addition, it is easy to see that $Ann_R(N)=Ann_R(f(N))$ and $Ann_R(N^\prime)=Ann_R(f^{-1}(N^\prime))$. Now, we use these assertions to prove the next two results.
%---------------------------------------------------------------------
%------------------------------Lemma 2.13-----------------------
%-------------------------------------------------------------------
\begin{lemma} \label{Lemma 2.13}
Let $f:M\rightarrow M^\prime$ be a module monomrphism of $R$-modules. Then the following hold:
\begin{enumerate}
\item If $N^\prime\in Spec^L(M^\prime)$ such that $N^\prime\subseteq f(M)$, then $f^{-1}(N^\prime)\in Spec^L(M)$.
\item If $N\in Spec^L(M)$, then $f(N)\in Spec^L(M^\prime)$.
\end{enumerate}
\end{lemma}
\begin{proof}
(1) Since $N^\prime\neq\{0_{M^\prime}\}$ and $N^\prime\subseteq f(M)$, then $f^{-1}(N^\prime)\neq\{0_M\}$. Now, for any $r\in R$, we have $rN^\prime=N^\prime$ or $r\in\sqrt{Ann_R(N^\prime)}$ as $N^\prime$ is a secondary submodule of $M^\prime$. Thus $rf^{-1}(N^\prime)=f^{-1}(rN^\prime)=f^{-1}(N^\prime)$ or $r\in\sqrt{Ann_R(N^\prime)}=\sqrt{Ann_R(f^{-1}(N^\prime)})$. This implies that $f^{-1}(N^\prime)$ is a secondary submodule of $M$. Since $Ann_R(soc(N^\prime))=\sqrt{Ann_R(N^\prime)}$ and $N^\prime\subseteq f(M)$, then $Ann_R(soc(f^{-1}(N^\prime)))=Ann_R(f^{-1}(soc(N^\prime)))=Ann_R(soc(N^\prime))=\sqrt{Ann_R(N^\prime)}=\sqrt{Ann_R(f^{-1}(N^\prime))}$. This means that $f^{-1}(N^\prime)\in Spec^L(M)$, as desired.\\
(2) It is easy to check that $f(N)$ is a secondary submodule of $M^\prime$. Now, since $N\in Spec^L(M)$, then $Ann_R(soc(N))=\sqrt{Ann_R(N)}$ which follows that $Ann_R(soc(f(N)))=Ann_R(f(soc(N)))=Ann_R(soc(N))=\sqrt{Ann_R(N)}=\sqrt{Ann_R(f(N))}$. Therefore $f(N)\in Spec^L(M^\prime)$, as needed.
\end{proof}
%-------------------------------------------------------------------------
%------------------------------Theorem 2.14---------------------------
%------------------------------------------------------------------------------
\begin{theorem} \label{Theorem 2.14}
Let $f:M\rightarrow M^\prime$ be a module monomorphism of $R$-modules. Then the map $\rho:Spec^L(M)\rightarrow Spec^L(M^\prime)$ defined by $\rho(K)=f(K)$ is an injective continuous map. Moreover, if $\rho$ is surjective, then $Spec^L(M)$ is homeomorphic to $Spec^L(M^\prime)$.
\end{theorem}
\begin{proof}
By Lemma \ref{Lemma 2.13}(2), $\rho$ is well-defined. In addition, it can easily be checked that $\rho$ is injective. By Lemma \ref{Lemma 2.4}(5) and (6), we obtain that for any $K\in Spec^L(M)$ and any closed set $\nu^s(N^\prime)$ in $Spec^L(M^\prime)$, where $N^\prime\leq M^\prime$, we have $K\in\rho^{-1}(\nu^s(N^\prime))=\rho^{-1}(\nu^{s*}(Ann_{M^\prime}(\sqrt{Ann_R(N^\prime)})))\Leftrightarrow soc(f(K))\subseteq Ann_{M^\prime}(\sqrt{Ann_R(N^\prime)})\Leftrightarrow\sqrt{Ann_R(N^\prime)}\subseteq Ann_R(soc(f(K)))=Ann_R(f(soc(K)))=Ann_R(soc(K))\Leftrightarrow soc(K)\subseteq Ann_M(\sqrt{Ann_R(N^\prime)})\Leftrightarrow K\in\nu^{s*}(Ann_M(\sqrt{Ann_R(N^\prime)}))=\nu^s(Ann_M(\sqrt{Ann_R(N^\prime)}))$. Thus $\rho^{-1}(\nu^s(N^\prime))=\nu^s(Ann_M(\sqrt{Ann_R(N^\prime)}))$ and hence $\rho$ is continuous. Now, suppose that $\rho$ is surjective. To complete the proof, it is sufficient to show that $\rho$ is closed. So let $\nu^s(N)$ be any closed set in $Spec^L(M)$, where $N\leq M$. As we have seen $\rho^{-1}(\nu^s(N^\prime))=\nu^s(Ann_M(\sqrt{Ann_R(N^\prime)}))$ for any submodule $N^\prime$ of $M^\prime$. So $\rho^{-1}(\nu^s(f(N)))=\nu^s(Ann_M(\sqrt{Ann_R(f(N))}))=\nu^s(Ann_M(\sqrt{Ann_R(N)}))=\nu^s(N)$ by Lemma \ref{Lemma 2.4}(6). Since $\rho$ is surjective, then $\rho(\nu^s(N))=\nu^s(f(N))$. This implies that $\rho$ is closed. Consequently, $Spec^L(M)$ is homeomorphic to $Spec^L(M^\prime)$.
\end{proof}
The proof of the following result is straightforward by using Lemma \ref{Lemma 2.13}(1) and Theorem \ref{Theorem 2.14}.
%-------------------------------------------------------------------------------
%----------------------------------Corollary 2.15------------------------------
%--------------------------------------------------------------------------
\begin{corollary}\label{corollary 2.15}
Let $f:M\rightarrow M^\prime$ be a module isomorphism of $R$-modules. Then $Spec^L(M)$ is homeomorphic to $Spec^L(M^\prime)$.
\end{corollary}
%--------------------------------------------------------------------------------------------------------------
%-------------------------------------------------------------------------------------------------------------------
\section{A base for the Zariski topology on $Spec^L(M)$}
Let $M$ be an $R$-module and let $D_r=Spec(R)-V^R(rR)$ for $r\in R$. It is known that the collection $\{D_r\,\mid\, r\in R\}$ is a base for the Zariski topology on $Spec(R)$. In addition, each $D_r$ is quasi-compact and thus $Spec(R)=D_1$ is quasi-compact.

In this section, we put $E_r=Spec^L(M)-\nu^s(Ann_M(r))$ for each $r\in R$ and show that $E=\{E_r\,\mid\,r\in R\}$ forms a base for the for the $\mathcal{SL}$-topology and give some related results. It is clear that $E_0=\emptyset$ and $E_1=Spec^L(M)$.
%-----------------------------------------------------------------------------
%-----------------------------Theorem 3.1---------------------------
%-------------------------------------------------------------
\begin{theorem}\label{theorem 3.1}
Let $M$ be an $R$-module. Then the set $E=\{E_r\,\mid\, r\in R\}$ forms a base for $Spec^L(M)$ with the $\mathcal{SL}$-topology.
\end{theorem}
\begin{proof}
Let $W=Spec^L(M)-\nu^s(N)$ be an open set in $Spec^L(M)$, where $N\leq M$ and let $K\in W$. To complete the proof, and it is sufficient to find $r\in R$ such that $K\in E_r\subseteq W$. As $K\in W$, then $Ann_R(N)\nsubseteq\sqrt{Ann_R(K)}=Ann_R(soc(K))$ and hence $\exists r\in Ann_R(N)-Ann_R(soc(K))$. If $K\notin E_r$, then $rR\subseteq Ann_R(Ann_M(r))\subseteq Ann_R(soc(K))$ and hence $r\in Ann_R(soc(K))$ which is a contradiction, and this means that $K\in E_r$. As $r\in Ann_R(N)$, then $Ann_M(Ann_R(N))\subseteq Ann_M(r)$. Now, we prove that $E_r\subseteq W$. So let $K^\prime\in E_r$. Then $Ann_R(Ann_M(r))\nsubseteq Ann_R(soc(K^\prime))$. If $K^\prime\notin W$, then $Ann_R(N)\subseteq Ann_R(soc(K^\prime))$ which follows that $Ann_M(Ann_R(soc(K^\prime)))\subseteq Ann_M(Ann_R(N))\subseteq Ann_M(r)$. So $Ann_R(Ann_M(r))\subseteq Ann_R(Ann_M(Ann_R(soc(K^\prime))))=Ann_R(soc(K^\prime))$, a contradiction. Therefore $K^\prime\in W$. Consequently, $K\in E_r\subseteq W$, as needed.
\end{proof}
%--------------------------------------------------------------------
%----------------------------------Proposition 3.2-----------------
%--------------------------------------------------------------------
For any ring $R$, the set of all units in $R$ and the nilradical of $R$ will be denoted by $U(R)$ and $N(R)$, respectively.
\begin{proposition} \label{Proposition 3.2}
Let $M$ be an $R$-module and $r\in R$. Then we have the following:
\begin{enumerate}
\item $\varphi^{-1}(D_{\overline{r}})=E_r$.
\item $\varphi(E_r)\subseteq D_{\overline{r}}$. If $\varphi$ is surjective, then the equality holds.
\item $E_a\cap E_b=E_{ab}$ for any elements $a, b\in R$.
\item If $r\in N(R)$, then $E_r=\emptyset$.
\item If $r\in U(R)$, then $E_r=Spec^L(M)$.
\end{enumerate}
\end{proposition}
\begin{proof}

(1)$\varphi^{-1}(D_{\overline{r}})=\varphi^{-1}(Spec(\overline{R})-V^{\overline{R}}(\overline{r}\overline{R}))=Spec^L(M)-\varphi^{-1}(V^{\overline{R}}(\overline{r}\overline{R}))=Spec^L(M)-\nu^s(Ann_M(r))=E_r$ by Proposition \ref{Proposition 2.9}. \\
(2) follows from (1).\\
(3) $E_a\cap E_b=\varphi^{-1}(D_{\overline{a}})\cap\varphi^{-1}(D_{\overline{b}})=\varphi^{-1}(D_{\overline{a}}\cap D_{\overline{b}})=\varphi^{-1}(D_{\overline{ab}})=E_{ab}$.\\
(4) Suppose that $r\in N(R)$. Then $D_r=\emptyset$ and hence $D_{\overline{r}}=\emptyset$. Thus $E_r=\varphi^{-1}(D_{\overline{r}})=\emptyset$.\\
(5) Suppose that $r\in U(R)$. Then $D_r=Spec(R)$ which follows that $D_{\overline{r}}=Spec(\overline{R})$. Therefore $E_r=\varphi^{-1}(D_{\overline{r}})=\varphi^{-1}(Spec(\overline{R}))=Spec^L(M)$.
\end{proof}
%---------------------------------------------------------
%----------------------------Corollary 3.3----------------
%-----------------------------------------------------------------------
\begin{corollary}\label{corollary 3.3}
Let $M$ be an $R$-module. If $R$ is field, then the Zariski topology on $Spec^L(M)$ is the trivial topology.
\end{corollary}
\begin{proof}
For any $r\in R-\{0\}$, we have $r\in U(R)$ and hence $E_r=Spec^L(M)$ by Proposition \ref{Proposition 3.2}(5). Also $E_0=\emptyset$. So $E=\{E_r\,\mid\, r\in R\}=\{Spec^L(M),\emptyset\}$ and so the Zariski topology on $Spec^L(M)$ is the trivial topology.
\end{proof}
The converse of the above result is not true in general. For example, take $M_2(\mathbb{Z}_8)$ (the group of all $2\times 2$ metrices with entries in $\mathbb{Z}_8$) as $\mathbb{Z}_8$-module. Then $1, 3, 5, 7\in U(\mathbb{Z}_8)$ and $0, 2, 4, 6\in N(\mathbb{Z}_8)$ which follows that $E_1=E_3=E_5=E_7=Spec^L(M_2(\mathbb{Z}_8))$ and $E_0=E_2=E_4=E_6=\emptyset$ by Proposition \ref{Proposition 3.2}. So $E=\{E_r\,\mid\, r\in R\}=\{\emptyset, Spec^L(M_2(\mathbb{Z}_8))\}$. Therefore the Zariski topology on $Spec^L(M_2(\mathbb{Z}_8))$ is the trivial topology. However $\mathbb{Z}_8$ is not field.
%------------------------------------------------------------------
%--------------------------------Theorem 3.4----------------------
%-----------------------------------------------------------------------------------
\begin{theorem}\label{Theorem 3.4}
Let $M$ be an $R$-module. If $\varphi$ is surjective, then for each $r\in R$, $E_r$ is quasi-compact. Therefore, $Spec^L(M)$ is quasi-compact.
\end{theorem}
\begin{proof}
Let $r\in R$ and $\mathcal{B}=\{E_a\,\mid\, a\in\Lambda\}$ be a basic open cover for $E_r$, where $\Lambda\subseteq R$. This follows that $E_r\subseteq \underset{a\in\Lambda}{\bigcup} E_a$ and thus $D_{\overline{r}}=\varphi(E_r)\subseteq\underset{a\in\Lambda}{\bigcup} \varphi(E_a)=\underset{a\in\Lambda}{\bigcup} D_{\overline{a}}$ using Proposition \ref{Proposition 3.2}(2). So $\overline{\mathcal{B}}=\{D_{\overline{a}}\,\mid\, a\in\Lambda\}$ is a basic open cover for $D_{\overline{r}}$ which is a quasi-compact set in $Spec(R)$, and so it has a finite subcover $\mathcal{B}^*=\{D_{\overline{a_j}}\,\mid\, j=1,..., m\}$, where $m\in \mathbb{N}$ and $t_j\in\Lambda$ for any $j=1,...,m$. Thus $D_{\overline{r}}\subseteq\bigcup\limits_{j=1}^m D_{\overline{a_j}}$ which implies that $E_r=\varphi^{-1}(D_{\overline{r}})\subseteq\bigcup\limits_{j=1}^m\varphi^{-1}(D_{\overline{a_j}})=\bigcup\limits_{j=1}^m E_{\overline{a_j}}$ using Proposition \ref{Proposition 3.2}(1). Therefore $\mathcal{B}^{**}=\{E_{\overline{a_j}}\,\mid\, j=1,...,m\}\subseteq \mathcal{B}$ is a finite subcover for $E_r$. The last part of the theorem follows from the equality $Spec^L(M)=E_1$.
\end{proof}
%----------------------------------------------------------------------
%-----------------------------Theorem 3.5----------------------------------
%---------------------------------------------------------------------------------
\begin{theorem}\label{Theorem 3.5}
Let $M$ be an $R$-module and let $\varphi$ be surjective. Then the quasi-compact open subsets of $Spec^L(M)$ are closed under finite intersections and form an open base.
\end{theorem}
\begin{proof}
Let $C_1, C_2$ be quasi-compact open sets of $Spec^L(M)$ and $\mathcal{B}=\{E_r\,\mid\, r\in\Delta\}$ be a basic open cover for $C_1\cap C_2$, where $\Delta\subseteq R$. By Theorem \ref{theorem 3.1}, $E$ is a base for $Spec^L(M)$ with the $\mathcal{SL}$-topology which follows that $C_1$ and $C_2$ can be written as a finite union of elements of $E$. Thus $C_1\cap C_2$ is a finite union of elements of $E$ using Proposition \ref{Proposition 3.2}(3). So there exists $m\in\mathbb{N}$ and $h_1, h_2, ..., h_m\in R$ such that $C_1\cap C_2=\bigcup\limits_{k=1}^m E_{h_k}\subseteq\underset{r\in\Delta}{\bigcup}E_r$. Hence $E_{h_k}\subseteq\underset{r\in\Delta}{\bigcup}E_r$ for each k. By Theorem \ref{Theorem 3.4}, each $E_{h_k}$ is quasi-compact which implies that $E_{h_k}\subseteq \bigcup\limits_{i=1}^{n_k}E_{r_{k, i}}$, where $n_k\geq 1$ depends on $k$ and $r_{k, i}\in\Delta$ for any $k=1, ..., m$ and  $i=1, ..., n_k$. This follows that $C_1\cap C_2=\bigcup\limits_{k=1}^m E_{h_k}\subseteq \bigcup\limits_{k=1}^m\bigcup\limits_{i=1}^{n_k}E_{r_{k, i}}$. Therefore $\mathcal{B}^*=\{E_{r_{k,i}}\,\mid\, k=1, ..., m, i=1, ..., n_k\}$ is a finite subcover for $C_1\cap C_2$. The other part of the theorem is obvious.
\end{proof}
%-----------------------------------------------------------------------------------------------------
%----------------------------------------------------------------------------------------------
\section{Irreducibility in $Spec^L(M)$ }
Let $M$ be an $R$-module and $Y\subseteq Spec^L(M)$. The closure of $Y$ in $Spec^L(M)$ will be denoted by $Cl(Y)$. We also denote the sum $\underset{K\in Y}{\sum} soc(K)$ by $H(Y)$. If $Y=\emptyset$, we set $H(Y)=\{0\}$.
%----------------------------------------------------------
%--------------------------------Proposition 4.1------------------------
%-----------------------------------------------------------------------------------
\begin{proposition}\label{proposition 4.1}
For any $R$-module $M$ and $Y\subseteq Spec^L(M)$, we have $Cl(Y)=\nu^s(H(Y))$. Therefore, $Y$ is closed in $Spec^L(M)$ $\Leftrightarrow$ $\nu^s(H(Y))=Y$. Moreover, if $M\in Y$, then $Y$ is dense in $Spec^L(M)$.
\end{proposition}
\begin{proof}
Clearly, $Y\subseteq\nu^s(H(Y))$. Let $\nu^s(N)$ be any closed set containing $Y$ and it is sufficient to show that $\nu^s(H(Y))\subseteq\nu^s(N)$. So let $K\in\nu^s(H(Y))$. Then $Ann_R(H(Y))\subseteq Ann_R(soc(K))$. But for any $K^\prime\in Y$, we have $Ann_R(N)\subseteq Ann_R(soc(K^\prime))$ and hence $Ann_R(N)\subseteq \underset{Q\in Y}{\bigcap}Ann_R(soc(Q))=Ann_R(\underset{Q\in Y}{\sum}soc(Q))=Ann_R(H(Y))\subseteq Ann_R(soc(K))$.  Thus $K\in\nu^s(H(Y))$. This means that $\nu^s(H(Y))$ is the smallest closed set containing $Y$ and hence $Cl(Y)=\nu^s(H(Y))$. For the last statement, if $M\in Y$, then $Cl(Y)=\nu^s(H(Y))=\nu^s(soc(M))=\nu^s(M)=Spec^L(M)$ using Lemma \ref{Lemma 2.4}(7).
\end{proof}
Recall that a topological space $Z$ is said to be irreducible if $Z\neq\emptyset$ and whenever $Z=Z_1\cup Z_2$ with closed sets $Z_1$ and $Z_2$ in $Z$, then either $Z=Z_1$ or $Z=Z_2$. Let $Z^\prime\subseteq Z$. Then $Z^\prime $ is irreducible if it is irreducible as a subspace of $Z$. The maximal irreducible subsets of $Z$ are called the irreducible components of $Z$. It is not difficult to prove that any singleton subset of $Z$ is irreducible. Also, a subset $W$ of $Z$ is irreducible if and only if $Cl(W)$ is irreducible, see \cite{bourbakialgebre}.
%------------------------------------------------------------------------------
%--------------------------------Theorem 4.2--------------------------------
%-------------------------------------------------------------------------------------------
\begin{theorem}\label{theorem 4.2}
Let $M$ be an $R$-module and let $K\in Spec^L(M)$. Then $Cl(\{K\})=\nu^s(K)$ and $\nu^s(K)$ is an irreducible closed subset of $Spec^L(M)$. In particular, if $M\in Spec^L(M)$, then $Spec^L(M)$ is irreducible.
\end{theorem}
\begin{proof}
By Lemma \ref{Lemma 2.4}(7) and Proposition \ref{proposition 4.1}, $Cl(\{K\})=\nu^s(H(K))=\nu^s(soc(K))=\nu^s(K)$. Now, since $\{K\}$ is irreducible, then $Cl(\{K\})=\nu^s(K)$ is irreducible. The last statement follows from the equality $Spec^L(M)=\nu^s(M)$.
\end{proof}
Note that $\nu^s(N)$ for a submodule $N$ of an $R$-module $M$ is not always irreducible. In fact, $Spec^L(M)$ might not be irreducible. For example, take $\mathbb{Z}_6$ as $\mathbb{Z}_6$-module. By some computations, we can see that $Spec^L(\mathbb{Z}_6)=\{\{0, 2, 4\}, \{0, 3\}\}, \nu^s(3\mathbb{Z}_6)=\{\{0, 3\}\}$ and $\nu^s(2\mathbb{Z}_6)=\{0, 2, 4\}$. Note that $Spec^L(\mathbb{Z}_6)=\nu^s(2\mathbb{Z}_6)\cup\nu^s(3\mathbb{Z}_6)$. But $Spec^L(\mathbb{Z}_6)\neq \nu^s(2\mathbb{Z}_6)$ and $Spec^L(\mathbb{Z}_6)\neq \nu^s(3\mathbb{Z}_6)$. Hence $Spec^L(\mathbb{Z}_6)=\nu^s(\mathbb{Z}_6)$ is not irreducible.
%--------------------------------------------------------------------------
%---------------------------------------------Corollary 4.3-------------
%-----------------------------------------------------------------------------
\begin{corollary}
Let $M$ be an $R$-module and $Y\subseteq Spec^L(M)$ such that $\sqrt{Ann_R(H(Y))}=p$ is a prime ideal of $R$. If $Spec^L_p(M)\neq\emptyset$, then $Y$ is irreducible in $Spec^L(M)$.
\end{corollary}
\begin{proof}
Let $K\in Spec^L_p(M)$. Then $\sqrt{Ann_R(K)}=p=\sqrt{Ann_R(H(Y))}$ and thus $\nu^s(K)=\nu^s(H(Y))=Cl(Y)$ by Lemma \ref{Lemma 2.4}(3) and Proposition \ref{proposition 4.1}. Using Theorem \ref{theorem 4.2}, $Cl(Y)=\nu^s(K)$ is irreducible and hence $Y$ is irreducible.
\end{proof}
Let $R$ be a ring and $Y\subseteq Spec(R)$. The intersection of all members of $Y$ will be given by $\xi(Y)$. If $Y=\emptyset$, we set $\xi(Y)=R$.
%----------------------------------------------------------------------------------
%---------------------------Theorem 4.4---------------------------------
%---------------------------------------------------------------------------------
\begin{theorem}\label{theorem 4.4}
Let $M$ be an $R$-module and $Y\subseteq Spec^L(M)$. Then:
\begin{enumerate}
\item If $H(Y)$ is a secondary submodule of $M$, then $Y$ is irreducible.
\item If $Y$ is irreducible, then $\Upsilon=\{Ann_R(soc(K))\,\mid\, K\in Y\}$ is an irreducible closed subset of $Spec(R)$, i.e. $\xi(\Upsilon)=Ann_R(H(Y))$ is a prime ideal of $R$.
\end{enumerate}
\end{theorem}
\begin{proof}
(1) Suppose that $H(Y)$ is a secondary submodule of $M$. By \cite[Proposition 2.1(h)]{ansari2013dual}, we have $soc(H(Y))\subseteq Ann_M(\sqrt{Ann_R(H(Y))})$ and hence $\sqrt{Ann_R(H(Y))}\subseteq Ann_R(Ann_M(\sqrt{Ann_R(H(Y))}))\subseteq Ann_R(soc(H(Y)))$. Thus $\sqrt{Ann_R(H(Y))}\subseteq Ann_R(soc(H(Y)))$. But $soc(H(Y))=H(Y)$. This follows that $Ann_R(soc(H(Y))=Ann_R(H(Y))\subseteq \sqrt{Ann_R(H(Y))}$. This means that $Ann_R(soc(H(Y)))=\sqrt{Ann_R(H(Y))}$ and hence $H(Y)\in Spec^L(M)$. By Theorem \ref{theorem 4.2}, $Cl(Y)=\nu^s(H(Y))$ is irreducible in $Spec^L(M)$ and thus $Y$ is irreducible, as desired.\\
(2) Suppose that $Y$ is irreducible in $Spec^L(M)$. So $\varphi(Y)=Y^\prime$ is an irreducible subset of $Spec(\overline{R})$ as $\varphi$ is continuous using Proposition \ref{Proposition 2.9}. Now, $\xi(Y^\prime)=\xi(\varphi(Y))=\underset{K\in Y}{\bigcap} \overline{Ann_R(soc(K))}=\overline{\underset{K\in Y}{\bigcap}Ann_R(soc(K))}=\overline{Ann_R(\underset{K\in Y}{\sum} soc(K))}=\overline{Ann_R(H(Y))}$. Since $Y^\prime$ is irreducible in $Spec(\overline{R})$, then $\xi(Y^\prime)=\overline{Ann_R(H(Y))}\in Spec(\overline{R})$ by \cite[p. 129, Proposition 14]{bourbakialgebre}. This follows that $\xi(\Upsilon)=\underset{K\in Y}{\bigcap} Ann_R(soc(K))=Ann_R(H(Y))\in Spec(R)$. Again, $\Upsilon$ is an irreducible subset of $Spec(R)$ by \cite[p. 129, Proposition 14]{bourbakialgebre}.
\end{proof}
Let $Z$ be a topological space and $F$ be a closed subset of $Z$. Recall that an element $a\in F$ is called a generic point of $F$ if $Cl(\{a\})=F$.
%---------------------------------------------------------------------------------------
%---------------------Theorem 4.5---------------------------------------------
%-------------------------------------------------------------------------------------
\begin{theorem} \label{theorem 4.5}
Let $M$ be an $R$-module and $Y\subseteq Spec^L(M)$. If $\varphi$ is surjective, then $Y$ is an irreducible closed subset of $Spec^L(M)$ if and only if $Y=\nu^s(K)$ for some $K\in Spec^L(M)$. Therefore, every irreducible closed subset of $Spec^L(M)$ has a generic point.
\end{theorem}
\begin{proof}
Suppose that $Y$ is an irreducible closed subdet of $Spec^L(M)$. Then $Ann_R(H(Y))\in Spec(R)$ by Theorem \ref{theorem 4.4}(2) and hence $\overline{Ann_R(H)}\in Spec(\overline{R})$. As $\varphi$ is surjective, we obtain $Ann_R(H(Y))=Ann_R(soc(K))$ for some $K\in Spec^L(M)$. This follows that $\sqrt{Ann_R(K)}=\sqrt{Ann_R(H(Y))}$ and thus $\nu^s(K)=\nu^s(H(Y))=Cl(Y)=Y$ by Lemma \ref{Lemma 2.4}(3) and Proposition \ref{proposition 4.1}. Conversely, if $Y=\nu^s(K)$ for some $K\in Spec^L(M)$, then $Y$ is closed and it is irreducible using Theorem \ref{theorem 4.2}.
\end{proof}
%---------------------------------------------------------------------------------
%---------------------------------Theorem 4.6--------------------
%-------------------------------------------------------------------------
\begin{theorem} \label{theorem 4.6}
Let $M$ be an $R$-module and $K\in Spec^L(M)$. If $\overline{Ann_R(soc(K))}$ is a minimal prime ideal of $\overline{R}$, then $\nu^s(K)$ is an irreducible component of $Spec^L(M)$. If $\varphi$ is surjective, then the converse is also true.
\end{theorem}
\begin{proof}
Using theorem\ref{theorem 4.2}, $\nu^s(K)$ is irreducible and it is sufficient to prove that it is a maximal irreducible. So let $Y$ be irreducible subset of $Spec^L(M)$ with $\nu^s(K)\subseteq Y$ and it remains to show that $Y=\nu^s(K)$. As $K\in\nu^s(K)\subseteq Y$, then $K\in Y$ and hence $soc(K)\subseteq H(Y)$ which implies that $\overline{Ann_R(H(Y))}\subseteq \overline{Ann_R(soc(K))}$. Using Theorem \ref{theorem 4.4}(2), we get $\overline{Ann_R(H(Y))}\in Spec(\overline{R})$. But $\overline{Ann_R(soc(K))}$ is a minimal prime ideal of $\overline{R}$. This follows that $\overline{Ann_R(H(Y))}=\overline{Ann_R(soc(K))}$ and hence $V^{\overline{R}}(\overline{Ann_R(H(Y))})=V^{\overline{R}}(\overline{Ann_R(soc(K))})$. By Lemma \ref{Lemma 2.4}(6), (7) and Proposition \ref{Proposition 2.9}, we get $\nu^s(H(Y))=\varphi^{-1}(V^{\overline{R}}(\overline{Ann_R(H(Y))}))=\varphi^{-1}(V^{\overline{R}}(\overline{Ann_R(soc(K))}))=\nu^s(soc(K))=\nu^s(K)$. Since $Y\subseteq\nu^s(H(Y))$, then $Y\subseteq\nu^s(K)$ and thus $Y=\nu^s(K)$. For the converse, suppose that $\varphi$ is surjective. Let $\overline{p}\in Spec(\overline{R})$ such that $\overline{p}\subseteq \overline{Ann_R(soc(K))}$ and it is sufficient to show that $\overline{p}=\overline{Ann_R(soc(K))}$. Since $\varphi$ is surjective, then there exists $K^\prime\in Spec^L(M)$ such that $Ann_R(soc(K^\prime))=p$. Then $Ann_R(soc(K^\prime))\subseteq Ann_R(soc(K))$ and hence $\nu^s(K)\subseteq\nu^s(K^\prime)$. Since $\nu^s(K)$ is an irreducible component of $Spec^L(M)$ and $\nu^s(K^\prime)$ is irreducible by Theorem \ref{theorem 4.2}, then $\nu^s(K)=\nu^s(K^\prime)$. By Lemma \ref{Lemma 2.4}(3), we obtain $\sqrt{Ann_R(K)}=\sqrt{Ann_R(K^\prime)}$ and thus $Ann_R(soc(K))=Ann_R(soc(K^\prime)$. Therefore, $\overline{p}=\overline{Ann_R(soc(K^\prime))}=\overline{Ann_R(soc(K))}$, as desired.
\end{proof}
%-------------------------------------------------------------------------------
%-----------------------Corollary 4.7---------------------------------
%-------------------------------------------------------------------------------
\begin{corollary}
Let $M$ be an $R$-module. If $\varphi$ is surjective, then the correspondence $\nu^s(K)\rightarrow\overline{Ann_R(soc(K))}$ provides a bijection from the set of irreducible components of $Spec^L(M)$ to the set of minimal prime ideals of $\overline{R}$.
\end{corollary}
\begin{proof}
The proof is straightforward by Theorem \ref{theorem 4.5} and Theorem \ref{theorem 4.6}.
\end{proof}
%---------------------------------------------------------------------------------------
%-------------------------------------------Corollary 4.8---------------------------
%--------------------------------------------------------------------------------------
\begin{corollary}\label{corollary 4.8}
Let $M$ be an $R$-module such that $Spec^L(M)\neq\emptyset$ and let $\mathcal{A}=\{K\in Spec^L(M)\,\mid\, \overline{Ann_R(soc(K))}$ is a minimal prime ideal of $\overline{R}\}$. If $\varphi$ is surjective, then we have the following:
\begin{enumerate}
\item $\mathcal{A}\neq\emptyset$.
\item $\mathcal{B}=\{\nu^s(K)\,\mid\, K\in\mathcal{A}\}$ is the set of all irreducible components of $Spec^L(M)$.
\item $Spec^L(M)=\underset{K\in \mathcal{A}}{\bigcup} \nu^s(K)$.
\item $Spec(\overline{R})=\underset{K\in\mathcal{A}}{\bigcup}V^{\overline{R}}(\overline{Ann_R(K)})$.
\item $Spec^s(M)=\underset{K\in \mathcal{A}}{\bigcup} V^s(K)$.
\item If $M\in Spec^L(M)$, then the only irreducible component of $Spec^L(M)$ is $Spec^L(M)$ itself.
\end{enumerate}
\end{corollary}
\begin{proof}
(1) Let $K\in Spec^L(M)$. Since any irreducible subset of a topological space is contained in a maximal irreducible subset of it, then there exists an irreducible component $Y$ in $Spec^L(M)$ such that $\{K\}\subseteq Y$. But every irreducible component of a topological space is closed and hence $Y=\nu^s(K^\prime)$ for some $K^\prime\in Spec^L(M)$ by Theorem \ref{theorem 4.5}. This follows that $\overline{Ann_R(soc(K^\prime))}$ is a minimal prime ideal of $\overline{R}$ by Theorem \ref{theorem 4.6}. Therefore, $K^\prime\in\mathcal{A}$, as desired.\\
(2) follows from Theorem \ref{theorem 4.5}, Theorem \ref{theorem 4.6} and the fact that the irreducible components of a topological space are closed subsets of it.\\
(3) Since the irreducible components of a topological space cover it, then $Spec^L(M)=\underset{K\in\mathcal{A}}{\bigcup}\nu^s(K)$ by (2).\\
(4) Since $Spec^L(M)=\underset{K\in\mathcal{A}}{\bigcup}\nu^s(K)$ and $\varphi$ is surjective, then $Spec(\overline{R})=\varphi(\underset{K\in\mathcal{A}}{\bigcup}\nu^s(K))=\underset{K\in\mathcal{A}}{\bigcup}\varphi(\nu^s(K))=\underset{K\in\mathcal{A}}{\bigcup}V^{\overline{R}}(\overline{Ann_R(K)})$ using Proposition \ref{Proposition 2.10}. \\
(5) Since $Spec(\overline{R})=\underset{K\in\mathcal{A}}{\bigcup}V^{\overline{R}}(\overline{Ann_R(K)})$, then $Spec^s(M)=\psi^{-1}(Spec(\overline{R}))=\psi^{-1}(\underset{K\in \mathcal{A}}{\bigcup}V^{\overline{R}}(\overline{Ann_R(K)}))=\underset{K\in\mathcal{A}}{\bigcup}\psi^{-1}(V^{\overline{R}}(\overline{Ann_R(K)}))=\underset{K\in\mathcal{A}}{\bigcup}V^s(Ann_M(Ann_R(K)))=\underset{K\in\mathcal{A}}{\bigcup}V^s(K)$ by Lemma \ref{Lemma 2.8}(1) and \cite[Lemma 3.3(b)]{farshadifar2013modules}.\\
(6) Suppose that $M\in Spec^L(M)$. Then $M$ is a secondary submodule of $M$ and hence $\overline{\sqrt{Ann_R(M)}}\in Spec(\overline{R})$. But for any $K\in\mathcal{A}$, we have $K\subseteq M$ and hence $\overline{\sqrt{Ann_R(M)}}\subseteq \overline{\sqrt{Ann_R(K)}}$. Thus $\sqrt{Ann_R(M)}=\sqrt{Ann_R(K)}$ as $\overline{\sqrt{Ann_R(K)}}$ is a minimal prime ideal of $\overline{R}$. This follows that $\nu^s(K)=\nu^s(M)=Spec^L(M)$ by Lemma \ref{Lemma 2.4}(3). So we get $\nu^s(K)=Spec^L(M)$ for any $K\in\mathcal{A}$. By (2), we have $\mathcal{B}=\{\nu^s(K)\,\mid\, K\in\mathcal{A}\}=\{Spec^L(M)\}$ is the set of all irreducible components of $Spec^L(M)$, as needed.
\end{proof}
Recall that a topological space $Z$ is said to be a $T_0$-space if the closure of distinct points are distinct. Note that $Spec^L(M)$ for an $R$-module $M$ is not always a $T_0$-space. For example, if $M$ is a vector space over a filed $F$ with $dim(M)>1$, then $\exists N_1, N_2\in Spec^s(M)\subseteq Spec^L(M)$ such that $N_1\neq N_2$. But by Corollary \ref{corollary 3.3}, the Zariski topology on $Spec^L(M)$ is the trivial topology. Since the trivial topology on any set $X$ is a $T_0$-space if and only if $|X|\leq 1$, then the Zariski topology on $Spec^L(M)$ is not a $T_0$-space. In fact, if $M$ is a vector space, then the Zariski topology on $Spec^L(M)$ is a $T_0$-space if and only if $dim(M)\leq 1$.
%----------------------------------------------------------------------------
%--------------------------------Theorem 4.9------------------------------
%-------------------------------------------------------------------------------
\begin{theorem}\label{theorem 4.9}
Let $M$ be an $R$-module. Then $Spec^L(M)$ is a $T_0$-space if and only if $|Spec^L_p(M)|\leq 1$ for any $p\in Spec(R)$.
\end{theorem}
\begin{proof}
Suppose that $Spec^L(M)$ is a $T_0$-space and let $K, K^\prime\in Spec^L_p(M)$. Then $\sqrt{Ann_R(K)}=\sqrt{Ann_R(K^\prime)}$ and thus $Cl(\{K\})=\nu^s(K)=\nu^s(K^\prime)=Cl(\{K^\prime\})$. Therefore, $K=K^\prime$. Conversely, let $K, K^\prime\in Spec^L(M)$ such that $K\neq K^\prime$ and assume by way of contradiction that $Cl(\{K\})=Cl(\{K^\prime\})$. Then $\nu^s(K)=\nu^s(K^\prime)$ and hence $\sqrt{Ann_R(soc(K))}=\sqrt{Ann_R(soc(K^\prime))}\in Spec(R)$. By hypothesis, $K=K^\prime$, a contradiction. This means that $Cl(\{K\})\neq Cl(\{K^\prime\})$. Therefore, $Spec^L(M)$ is a $T_0$-space.
\end{proof}
The following is an easy result for Proposition \ref{Proposition 2.6} and Theorem \ref{theorem 4.9}.
%--------------------------------------------------------------------------
%-----------------------------------Corollary 4.10-----------------------
%-----------------------------------------------------------------------
\begin{corollary}\label{corollary 4.10}
Let $M$ be an $R$-module. Then the following statements are equivalent:
\begin{enumerate}
\item $Spec^L(M)$ is a $T_0$-space.
\item $|Spec^L_p(M)|\leq 1$ for any $p\in Spec(R)$.
\item If $\nu^s(K_1)=\nu^s(K_2)$, then $K_1=K_2$ for any $K_1, K_2\in Spec^L(M)$.
\item The natural map $\varphi$ is injective.
\end{enumerate}
\end{corollary}
Let $Z$ be a topological space. Recall that $Z$ is said to be a spectral space if it is homemorphic to the prime spectrum of a ring equipped with the Zariski topology. By Hochester's Characterization of spectral spaces\cite[p.52, Proposition 4]{hochster1969prime}, $Z$ is a spectral space if and only if:
\begin{enumerate}
    \item[(a)]$Z$ is a $T_0$-space.
    \item[(b)]$Z$ is quasi-compact.
    \item[(c)]Any irreducible closed subset of $Z$ has a generic point.
    \item[(d)]The quasi-compact open subsets of $Z$ are closed under finite intersection and form an open base.
\end{enumerate}
%--------------------------------------------------------------------------------------
%--------------------------------------------------------Theorem 4.11-----------------
%----------------------------------------------------------------------------------
\begin{theorem}\label{theorem 4.11}
Let $M$ be an $R$-module. If $\varphi$ is surjective, then $Spec^L(M)$ is a spectral space if and only if $Spec^L(M)$ is a $T_0$-space.
\end{theorem}
\begin{proof}
By Theorem \ref{Theorem 3.4}, Theorem \ref{Theorem 3.5} and Theorem \ref{theorem 4.5}.
\end{proof}
The following result is obtained by combining Corollary \ref{corollary 2.11}, Corollary \ref{corollary 4.10} and Theorem \ref{theorem 4.11}.
%---------------------------------------------------Corollary 4.12-------------------
%-------------------------------------------------------------------------------
%---------------------------------------------------------------------
\begin{corollary}\label{corollary 4.12}
Let $M$ be an $R$-module. If $\varphi$ is surjective, then the following statements are equivalent:
\begin{enumerate}
\item $Spec^L(M)$ is a spectral space.
\item $Spec^L(M)$ is a $T_0$-space.
\item $|Spec^L_p(M)|\leq 1$ for every $p\in Spec(R)$.
\item $\varphi$ is injective.
\item $Spec^L(M)$ is homeomorphic to $Spec(\overline{R})$ under $\varphi$.
\end{enumerate}
\end{corollary}
Let $M$ be an $R$-module. By Corollary \ref{corollary 4.12}, the surjectivity of $\varphi$ implies that $Spec^L(M)$ is a spectral space if and only if $|Spec^L_p(M)|\leq 1$ for any $p\in Spec(R)$. In the following proposition, we replace the surjectivity of $\varphi$ by the finiteness of $Spec^L(M)$.
%---------------------------------------------------------------------------
%------------------------------------Proposition 4.13------------------------
%--------------------------------------------------------------------------
\begin{proposition}
Let $M$ be an $R$-module such that $Spec^L(M)$ is a non-empty finite set. Then $Spec^L(M)$ is a spectral space if and only if $|Spec^L_p(M)|\leq 1$ for every $p\in Spec(R)$.
\end{proposition}
\begin{proof}
Since $Spec^L(M)$ is finite, then (b) and (d) of Hochester's characterization of spectral spaces are satisfied. Now, let $F=\{a_1, a_2, ..., a_n\}$ be an irreducible closed subset of $Spec^L(M)$. Since $F$ is closed, then $F=Cl(\{a_1\}\cup\{a_2\}\cup ...\cup\{a_n\})=Cl(\{a_1\})\cup Cl(\{a_2\})\cup...\cup Cl(\{a_n\})$. Since $F$ is irreducible, then $F=Cl(\{a_i\})$ for some $i\in\{1, ..., k\}$. So every irreducible closed subset of $Spec^L(M)$ has a generic point, i.e. condition (c) of Hochester's characterization is also satisfied. Now, $Spec^L(M)$ is a spectral space if and only if $Spec^L(M)$ is a $T_0$-space if and only if $|Spec_p^L(M)|\leq 1$ for any $p\in Spec(R)$ by Theorem \ref{theorem 4.9}.
\end{proof}
Recall that a topological space $Z$ is said to be a $T_1$-space if every singleton subset of $Z$ is closed. Now, we need the next two lemmas to investigate $Spec^L(M)$ with the Zariski topology from the viewpoint of being a $T_1$-space.
%----------------------------------------------------------------------------------
%--------------------------------Lemma 4.14------------------------------------
%------------------------------------------------------------------------------------
\begin{lemma}\label{Lemma 4.14}
Let $M$ be an $R$-module such that every secondary submodule of $M$ contains a minimal submodule and let $K\in Spec^L(M)$. Then the set $\{K\}$ is closed in $Spec^L(M)$ if and only if $K$ is a minimal submodule of $M$ and $Spec_p^L(M)=\{K\}$, where $p=\sqrt{Ann_R(K)}$.
\end{lemma}
\begin{proof}$\Rightarrow$: Suppose that $\{K\}$ is closed in $Spec^L(M)$. Then $Cl(\{K\})=\{K\}$ and thus $\nu^s(K)=\{K\}$. By hypothesis, there exists a minimal submodule $T$ of $M$ such that $T\subseteq K$. Hence $\sqrt{Ann_R(K)}\subseteq\sqrt{Ann_R(T)}$. Since $T$ is a minimal submodule of $M$, then $T\in Spec^s(M)\subseteq Spec^L(M)$ by \cite[Proposition 1.6]{yassemi2001dual}. Thus $T\in\nu^s(K)=\{K\}$. This implies that $T=K$ and thus $K$ is a minimal submodule of $M$. Now, if $K^\prime\in Spec_p^L(M)$, then $\sqrt{Ann_R(K^\prime)}=\sqrt{Ann_R(K)}=p$ and hence $K^\prime\in\nu^s(K)=\{K\}$. This follows that $Spec_p^L(M)\subseteq\{K\}$. Consequently, $Spec^L_p(M)=\{K\}$.\\
$\Leftarrow$: Suppose that $B\in Cl(\{K\})$. Then $B\in\nu^s(K)$ and thus $Ann_R(K)\subseteq \sqrt{Ann_R(B)}$. Since $K$ is a minimal submodule of $M$, then $Ann_R(K)$ is a maximal ideal of $R$ which follows that $Ann_R(K)=\sqrt{Ann_R(B)}$. So $p=\sqrt{Ann_R(K)}=\sqrt{Ann_R(B)}$. But $Spec^L_p(M)=\{K\}$. Thus $B=K$ and hence $Cl(\{K\})\subseteq K$. Therefore $\{K\}$ is closed in $Spec^L(M)$, as desired.
\end{proof}
%-----------------------------------------------------------------------------------------
%--------------------------------------------------------Lemma 4.15------------------
%------------------------------------------------------------------------------
\begin{lemma}\label{lemma 4.15}
Let $M$ be an $R$-module and $p\in Spec(R)$. If $K_1, K_2\in Spec_p^L(M)$, then $K_1+K_2\in Spec_p^L(M)$.
\end{lemma}
\begin{proof}
Firstly, note that $\sqrt{Ann_R(K_1+K_2)}=\sqrt{Ann_R(K_1)\cap Ann_R(K_2)}=\sqrt{Ann_R(K_1)}\cap\sqrt{Ann_R(K_2)}=p$. Now, we show that $K_1+K_2$ is a secondary submodule of $M$. So let $r\in R$ and suppose that $r\notin\sqrt{Ann_R(K_1+K_2)}=\sqrt{Ann_R(K_1)}=\sqrt{Ann_R(K_2)}=p$. Since $K_1$ and $K_2$ are secondary submodules of $M$, then $rK_1=K_1$ and $rK_2=K_2$. This follows that $r(K_1+K_2)=K_1+K_2$. Now, it remains to prove that $Ann_R(soc(K_1+K_2))=\sqrt{Ann_R(K_1+K_2)}$. By \cite[Proposition 2.1(h)]{ansari2013dual}, $soc(K_1+K_2)\subseteq Ann_M(\sqrt{Ann_R(K_1+K_2)})$ and hence $p=\sqrt{Ann_R(K_1+K_2)}\subseteq Ann_R(Ann_M(\sqrt{Ann_R(K_1+K_2)}))\subseteq Ann_R(soc(K_1+K_2))$. Since $K_1\subseteq K_1+K_2$, then $soc(K_1)\subseteq soc(K_1+K_2)$ and hence $Ann_R(soc(K_1+K_2))\subseteq Ann_R(soc(K_1))=p=\sqrt{Ann_R(K_1+K_2)}$, as needed.
\end{proof}
%-------------------------------------------------------------------------------
%---------------------------------Corollary 4.16-------------------------
%----------------------------------------------------------------------------
\begin{corollary}
Let $M$ be an $R$-module such that every secondary submodule of $M$ contains a minimal submodule. Then $Spec^L(M)$ is a $T_1$-space if and only if $Min(M)=Spec^L(M)$, where $Min(M)$ is the set of all minimal submodules of $M$.
\end{corollary}
\begin{proof}
$\Rightarrow$: Suppose that $Spec^L(M)$ is a $T_1$-space. By \cite[Proposition 1.6]{yassemi2001dual}, $Min(M)\subseteq Spec^s(M)\subseteq Spec^L(M)$. Now, let $K\in Spec^L(M)$. Since $Spec^L(M)$ is a $T_1$-space, then $\{K\}$ is closed. By Lemma \ref{Lemma 4.14}, we obtain $K\in Min(M)$.\\
$\Leftarrow$: Suppose that $Min(M)=Spec^L(M)$ and let $K\in Spec^L(M)$. It is enough to show that $\{K\}$ is closed in $Spec^L(M)$. Let $p=\sqrt{Ann_R(K)}$ and $H\in Spec_p^L(M)$. By Lemma \ref{lemma 4.15}, $K+H\in Spec_p^L(M)\subseteq Min(M)$ and thus $K+H\in Min(M)$. But $H\subseteq H+K$ and $W\subseteq H+K$. So $H=H+K=K$ and so $Spec_p^L(M)=\{K\}$. By Lemma \ref{Lemma 4.14}, we have $\{K\}$ is closed in $Spec^L(M)$.
\end{proof}

%---------------------acknowledgment-------------------------------------------------
%----------------------------------------------------------------------------------

%
\noindent\textbf{Declarations}

%\bigskip\bigskip
%%\noindent\textbf{Author contribution statement:}\\
%K. Al-Zoubi : Conceived and designed the analysis; Performed the experiments; Analyzed and interpreted the data;  Contributed reagents, materials, analysis tools or data; Wrote %the paper.
%\smallskip

\noindent\textbf{Funding statement:}\\ This research did not receive any specific grant from funding agencies in the public, commercial, or not-for-profit sectors.

\smallskip

\noindent\textbf{Data availability statement:} \\No data was used for the research described in the article
\smallskip

\noindent\textbf{Declaration of interests statement: }\\The authors declare no conflict of interest
\smallskip

\noindent\textbf{Additional information: } \\No additional information is available for this paper.
\smallskip

\noindent\textbf{Acknowledgments:} \\The authors wish to thank sincerely the referees for
their valuable comments and suggestions.
\smallskip

%-----------------------------------------------------------------------------------------

\bigskip\bigskip\bigskip\bigskip

\end{document}